\documentclass[11pt]{amsart}
\usepackage{latexsym}
\usepackage{aeguill}
\usepackage[latin1]{inputenc}
\usepackage[T1]{fontenc}
\usepackage{amsfonts}
\usepackage{amssymb}
\usepackage{xspace}
\usepackage{graphicx}
\usepackage{amsmath}
\usepackage{amsthm}
\usepackage{endnotes}
\usepackage{tikz}
\usepackage{pgf}
\usepackage{yfonts}

\usepackage[all]{xy}
\usepackage{amscd}
\usetikzlibrary{matrix}

\newcommand{\cal}{\mathcal}

\def\epsilon{\varepsilon}
\def\phi{\varphi}

\def\hat{\widehat}

\newcommand{\Out}{\mbox{Out}}
\newcommand{\Aut}{\mbox{Aut}}

\newcommand{\FN}{F_n}   

\newcommand{\Z}{\mathbb Z}

\newcommand{\N}{\mathbb N}

\def\strutdepth{\dp\strutbox}
\def \ss{\strut\vadjust{\kern-\strutdepth \sss}}
\def \sss{\vtop to \strutdepth{
\baselineskip\strutdepth\vss\llap{$\diamondsuit\;\;$}\null}}

\def\strutdepth{\dp\strutbox}
\def \sst{\strut\vadjust{\kern-\strutdepth \ssss}}
\def \ssss{\vtop to \strutdepth{
\baselineskip\strutdepth\vss\llap{$\spadesuit\;\;$}\null}}

\def\strutdepth{\dp\strutbox}
\def \ssh{\strut\vadjust{\kern-\strutdepth \sssh}}
\def \sssh{\vtop to \strutdepth{
\baselineskip\strutdepth\vss\llap{$\heartsuit\;\;$}\null}}

\def\qed{\hfill\rlap{$\sqcup$}$\sqcap$\par}
\def\bar{\overline}
\def\tilde{\widetilde}

\def\strutdepth{\dp\strutbox}
\def \ss{\strut\vadjust{\kern-\strutdepth \sss}}
\def \sss{\vtop to \strutdepth{
\baselineskip\strutdepth\vss\llap{$\diamondsuit\;\;$}\null}}

\def\strutdepth{\dp\strutbox}
\def \sst{\strut\vadjust{\kern-\strutdepth \ssss}}
\def \ssss{\vtop to \strutdepth{
\baselineskip\strutdepth\vss\llap{$\spadesuit\;\;$}\null}}

\def\qed{\hfill\rlap{$\sqcup$}$\sqcap$\par}

\newtheorem*{thm*}{Theorem}
\newtheorem*{claim*}{Claim}

\newtheorem{thm}{Theorem}[section]
\newtheorem{cor}[thm]{Corollary}
\newtheorem{lem}[thm]{Lemma}
\newtheorem{prop}[thm]{Proposition}

\theoremstyle{definition}
\newtheorem{defn}[thm]{Definition}
\newtheorem{defn-rem}[thm]{Definition-Remark}

\newtheorem{rem}[thm]{Remark}

\theoremstyle{remark}

\numberwithin{equation}{section}

\begin{document}

\author[K.~Ye]{Kaidi Ye}

\title[Partial Dehn twists -  a dichotomy]{Partial Dehn twists of free groups relative to local Dehn twists - a dichotomy}

\begin{abstract} 
A criterion for quadratic or higher growth of group automorphisms is established which are represented by graph-of-groups automorphisms with certain well specified properties.

As a consequence, it is derived (using results of a previous paper of the author) that every partial Dehn twist automorphism of $\FN$ relative to local Dehn twist automorphisms is either an honest Dehn twist automorphism, or else has quadratic growth.
\end{abstract}

\subjclass[1991]{Primary 20F, Secondary 20E}
\keywords{graph-of-groups, free group, Dehn twists, Combinatorics of words}

\maketitle

\section{Introduction}

Dehn twist are well known from surface homeomorphisms: Any set $\cal C$ of pairwise disjoint essential closed curves $c_i$ on a surface $S$, together with a set of {\em twist exponents} $n_i \in \Z$, defines a homeomorphism $h: S \to S$ through ``twisting'' $S$ along each $c_i$ precisely $n_i$ times.

The set $\cal C$ defines canonically a {\em dual}
graph-of-groups $\cal G$ with isomorphism $\pi_1 \cal G \cong \pi_1 S$, 
where the vertex groups of $\cal G$ are the fundamental groups of the components of $S \smallsetminus \cal C$ and the edge groups are isomorphic to $\Z$. 
The automorphism $h_*: \pi_1 S \to \pi_1 S$ induced by the multi-Dehn-twist $h$ can be described algebraically through a graph-of-groups isomorphism $H: \cal G \to \cal G$.

This natural correspondence between geometric and algebraic data has given rise to a more general definition of a {\em Dehn twist automorphism} $\phi$ of a group $G$, via a graph-of-groups $\cal G$ equipped with an isomorphism $\pi_1 \cal{G} \cong G$ and a graph-of-groups isomorphism $H: \cal G \to \cal G$, which satisfy extra conditions that mimic the above described surface situation, so that one gets
$H_* = \phi$ (up to inner automorphisms).

A special case, which is useful in many circumstances, is given by requiring that $H$ acts trivially on the underlying graph $\Gamma(\cal G)$ and on each of the vertex groups $G_v$ of $\cal G$, and that furthermore all edge groups of $\cal G$ are trivial. In this case the automorphism $H_*: \pi_1 \cal G \to \pi_1 \cal G$ is determined by the family of {\em correction terms} $\delta_e \in G_{\tau(e)}$ for any edge $e$ of $\cal G$, where $\tau(e)$ denotes the terminal vertex of $e$.

\medskip

In the paper \cite{KY01} the more general case of a {\em partial Dehn twist $H: \cal G \to \cal G$ relative to a subset $\cal V$ of vertices of $\cal G$} 
has been investigated, which differs from the above described situation in that
for any $v \in \cal V$ the induced vertex group automorphism $H_v: G_v \to G_v$ may not be the identity.

If for each $v \in \cal V$ the map $H_v: \cal G_v \to \cal G_v$ is itself a 
Dehn twist automorphism, then $H: \cal G \to\cal G$ is called a
{\em partial Dehn twist relative to a family of local Dehn twists}. It is shown in \cite{KY01} that in this case $H$ can be {\em blown-up} to a refined graph-of-groups isomorphism which is a Dehn twist that incorporates both, $H$ and the family of all $H_v$, provided that the following criterion is satisfied:

\medskip
\noindent
{\em Criterion:} For every edge $e$ of $H$ with endpoint $v \in \cal V$ the correction term $\delta_e$ is $H_v$-zero.

\medskip

Here for any graph-of-groups automorphism $H: \cal G \to \cal G$,
the associated path group $\Pi(\cal G)$, and any vertex $v$ of $\cal G$
an element $g \in \pi_1 (\cal G, v) \subset \Pi(\cal G)$ is $H$-zero if and only if there exists an element $h \in \Pi(\cal G)$ such that $h^{-1} g H_{*v}(h)$ has $\cal G$-length 0.

The main result of this paper is to show that this sufficient criterion is also necessary. In fact, we show:

\begin{thm}
\label{q-growth}
Let $H: \cal G \to \cal G$ be a partial Dehn twist relative to 
a subset $\cal V$ of vertices of $\cal G$.
Assume for some $v \in \cal V$ that 
the vertex group $G_v$ is free, and that
$H_v: G_v \to G_v$ is a
Dehn twist automorphisms.

If there is an edge $e$ of $\cal G$ with correction term $\delta_e \in G_v$ that is not $H_v$-zero, then the automorphism $H_*: \pi_1 \cal G \to \pi_1 \cal G$ has at least quadratic growth.
\end{thm}

Since Dehn twist automorphisms are known to have linear growth, this shows that $H_*$ is not conjugate to any Dehn twist automorphism. In particular, $H$  can indeed not be blown up via the local Dehn twists $H_v$ to obtain a global Dehn twist of $\pi_1 \cal G$.

By combining this theorem with the main result of \cite{KY01}, we obtain:

\begin{cor}
\label{q-growth+}
Let $\hat \phi \in Out(\FN)$ be represented by a partial Dehn twist 
relative to a family of local Dehn twists.

Then either $\hat \phi$ is itself a Dehn twist automorphism, or else $\hat \phi$ has at least quadratic growth.
\end{cor}

The proof of this corollary is algorithmic, i.e. it can be effectively decided which alternative of the stated dichotomy holds. This is a crucial ingredient in the author's work \cite{KY}, where an algorithm is given that decides whether a polynomial growth automorphism of a free group $F_n$ is, up to passing to a power, induced by a surface homeomorphism.  It is also the starting point of a more detailed analysis of the growth of conjugacy classes for polynomially growing automorphisms of $F_n$, see \cite{KY}.

\subsection*{Acknowledgements}
${}^{}$

This paper is part of my PhD thesis, 
and I would like to thank sincerely my advisors, Arnaud Hilion and Martin Lustig, 
for their advice and encouragement.

\section{Graphs-of-groups and their isomorphsms}


The purpose of this and the following section is to briefly recall some basic knowledge 
and to establish some preliminary lemmas about 
graph-of-groups, Dehn twists on graph-of-groups, 
efficient Dehn twist as well as the notion of $H$-conjucation, 
which is introduced in \cite{KY01}. 

\subsection{Basics on graphs-of-groups}\label{graphofgroups}
${}^{}$

Most of our notations are taken from \cite{CL02}; 
we refer the readers to \cite{Serre}, \cite{RW} and \cite{Bass} for more detailed informations and discussions.

Throughout this paper, 
we refer to a {\it graph} as a finite, non-empty, connected graph in the sense of Serre (cf. \cite{Serre}). 

For a graph $\Gamma$, 
we denote by $V(\Gamma)$, $E(\Gamma)$ its {\it vertex set} and {\it edge set} respectively.
For an edge $e \in E(\Gamma)$, 
we deonote by $\tau(e)$ its {\it terminal vertex} and $\tau(\bar{e})$ its {\it initial vertex}.
The {\it inverse} of an edge $e$ is denoted by $\bar{e}$.

Notice in particular that our graph $\Gamma$ is non-oriented 
while one can always choose an {\it orientation} $E^{+}(\Gamma) \subset E(\Gamma)$, 
which satisfies that 
$E^{+}(\Gamma) \cup \bar{E}^{+}(\Gamma)=E(\Gamma)$ 
and 
$E^{+}(\Gamma) \cap \bar{E}^{+}(\Gamma)=\emptyset$, 
where $\bar{E}^{+}(\Gamma)=\{\bar{e} \mid e \in E^{+}(\Gamma)\}$.

\begin{defn}\label{graph of groups}
A {\it graph-of-groups} $\cal{G}$ is defined by
$$ \cal{G}=(\Gamma, (G_{v})_{v\in V(\Gamma)}, (G_{e})_{e \in E(\Gamma)}, (f_{e})_{e\in E(\Gamma)})$$
where:
\begin{enumerate}
\item $\Gamma$ is a graph, called the {\it underlying} graph;
\item each $G_{v}$ is a group, called the {\it vertex group} of $v$;
\item each $G_{e}$ is a group, called the {\it edge group} of $e$, 
and we require $G_{e}=G_{\bar{e}}$ for every $e \in E(\Gamma)$;
\item for each $e\in E(\Gamma)$,
the map 
$f_{e}: G_{e} \rightarrow G_{\tau(e)}$ is an injective {\it edge homomorphism}.
\end{enumerate}
\end{defn}

Unless otherwise stated, in this paper we will always assume that all vertex and all edge groups of any graph-of-groups $\cal G$ are finitely generated.

Given a graph-of-groups $\cal{G}$, 
we usually denote by $\Gamma(\cal{G})$ the graph underlying it. 
The vertex set of $\Gamma(\cal{G})$ is denoted by $V(\cal{G})$ 
while the edge set is denoted by $E(\cal{G})$.

\begin{defn} \label{bassgroup}
The {\it word group} $W(\cal{G})$ of a graph-of-groups $\cal{G}$ is 
the free product of vertex groups and the free group generated by {\it stable letters} $(t_{e})_{e \in E(\Gamma)}$, 
i.e. $W(\cal{G})=*(G_{v})_{v \in V(\Gamma)}*F(\{t_{e}; e\in E(\Gamma)\})$.

The {\it path group} 
(sometimes also called {\it Bass group}) 
of $\cal{G}$ is defined by $\Pi(\cal{G})= W(\cal{G})/ R$, 
where $R$ is the normal subgroup subjects to the following relations:
\begin{enumerate}
\item[$\diamond$] $t_{e}=t_{\bar{e}}^{-1}$, for every $e \in E(\Gamma)$;
\item[$\diamond$] $f_{\bar{e}}(g)=t_{e}f_{e}(g)t_{\bar{e}}^{-1}$, 
for every $e \in E(\Gamma)$ and every $g\in G_{e}$.
\end{enumerate}
\end{defn}

\begin{rem}
A {\it word} $w \in W(\cal{G})$ can always be written in the form 
$w = r_{0}t_{1}r_{1}...r_{q-1}t_{q}r_{q}$ ($q \geq 0$), 
where each $t_{i} \in F(\{t_{e}; 
e \in E(\Gamma)\})$ stands for the stable letter of the edge $e_i$ 
and each $r_{i} \in *(G_{v})_{v \in V(\Gamma)}$.

The sequence $(t_{1}, t_{2}, ..., t_{q})$ is called the {\it path type} of $w$, 
the number $q$ is called the {\it path length} of $w$. 
In this case, we say that $e_{1}e_{2}...e_{q}$ is the path underlying $w$. 
Two path types 
$(t_{1}, t_{2}, ..., t_{q})$ 
and $(t^{\prime}_{1}, t^{\prime}_{2}, ..., t^{\prime}_{s} )$ 
are said to be same if and only if $q=s$ and $t_{i}=t^{\prime}_{i}$ for each $1 \leq i \leq q$.
\end{rem}

\begin{defn}
Let $w \in W(\cal{G})$ be a word of the form $w = r_{0}t_{1}r_{1}...r_{q-1}t_{q}r_{q}$. 
The word $w$ is said to be {\it connected} if $r_{0} \in G_{\tau(\bar{e}_{1})}$, 
$r_{q} \in G_{\tau(e_{q})}$, 
and 
$\tau(e_{i})=\tau(\bar{e}_{i+1})$, 
$r_{i} \in G_{\tau(e_{i})}$, 
for $i=1,2,...,q-1$.

Moreover, if $w$ is connected and $\tau(e_{q})=\tau(\bar{e}_{1})$, 
we say that $w$ is a {\it closed connected word issued at the vertex $\tau(e_{q})$}.
\end{defn}

\begin{defn}
Let $w=r_{0}t_{1}r_{1}...r_{q-1}t_{q}r_{q} \in W(\cal{G})$, 
$w$ is said to be {\it reduced} if it satisfies:
\begin{enumerate}
\item[$\diamond$] if $q=0$, then $w = r_{0}$ isn't equal to the unit element;
\item[$\diamond$] if $q>0$, then whenever $t_{i} = t_{i+1}^{-1}$ for some $1 \leq i \leq q-1$ we have $r_{i}\not\in f_{e_{i}}(G_{e_{i}})$.
\end{enumerate}

Moreover the word $w$ is said to be {\it cyclically reduced} if it is reduced and
if $q>0$ and $t_{1} = t_{q}^{-1}$, then $r_{q}r_{0}\not\in f_{e_{q}}(G_{e_{q}})$.
\end{defn}

We recall the following facts.
\begin{prop}
\label{reduced-words-existence}
For any graph-of-groups $\cal{G}$, the following holds:
\begin{enumerate}
\item Every non-trivial element of $\Pi(\cal{G})$ can be represented as a reduced word.
\item Every reduced word is a non-trivial element in $\Pi(\cal{G})$.
\item If $w_{1},w_{2} \in W(\cal{G})$ are two reduced words representing the same element in $\Pi(\cal{G})$, 
then $w_{1}$ and $w_{2}$ are of the same path type. 
In particular, $w_{2}$ is connected if and only if $w_{1}$ is connected.
\end{enumerate}

\end{prop}

\begin{defn}(fundamental groups)
\begin{enumerate}
\item[1.] Fundamental groups based at $v_0$

For any $v_{0} \in V(\Gamma)$,
the {\it fundamental group based at $v_{0}$}, 
denoted by $\pi_{1}(\cal{G}, v_{0})$, 
consists of the elements in $\Pi(\cal{G})$ that are closed connected words issued at $v_{0}$.

For a vertex $w_{0} \in V(\Gamma)$ different from $v_{0}$, 
we have $\pi_{1}(\cal{G},v_{0}) \cong \pi_{1}(\cal{G},w_{0})$. 
In fact,
let $W \in \Pi(\cal{G})$ be a connected word with underlying path from $v_{0}$ to $w_{0}$. 
The restriction of 
$ad_{W} : \Pi(\cal{G}) \rightarrow \Pi(\cal{G})$ to $\pi_{1}(\cal{G},w_{0})$ 
induces an isomorphism from $\pi_{1}(\cal{G},w_{0})$ to $\pi_{1}(\cal{G},v_{0})$. 
Sometimes we write $\pi_{1}(\cal{G})$ when the choice of basepoint doesn't make a difference.

\item[2.] Fundamental groups at a maximal tree $T_0$

The {\it fundamental group at $T_{0}$}, 
denoted by $\pi_{1}(\cal{G}, T_{0})$, 
is generated by the groups $G_v$, for all $v \in V(\Gamma)$,
and the elements $t_e$, for all $e \in E(\Gamma)$,
subjects to the relations:
\begin{enumerate}
\item[$\diamond$]
$t_e^{-1}=t_{\bar{e}}$, $t_e f_{e}(g) t_{e}^{-1}=f_{\bar{e}}(g)$, for $e \in E(\Gamma)$, $g \in G_{e}$;
\item[$\diamond$]
$t_e = 1$, for $e \in E(T_0)$.
\end{enumerate}
By defintion we have immediately 
$$\pi_{1}(\cal{G}, T_0)= \Pi(\cal{G}) / \ll t_e=1, e \in E(T_0)\gg.$$

\end{enumerate}
\end{defn}

It's shown in the book of Serre \cite{Serre} that the above two definitions of fundamental groups are equivalent.

\begin{thm}[Serre]
For a graph-of-groups $\cal{G}$, let $v_0$ be a vertex and $T_0$ a maximal tree.
Then $\pi_{1}(\cal{G},v_0) \cong \pi_{1}(\cal{G}, T_0)$.
\end{thm}

It follows immediately that, for a graph-of-groups $\cal{G}$ with trivial edge groups, the product $*(G_{v})_{v\in V(\cal{G})}$ is free and forms a free factor of $\pi_{1}(\cal{G})$, moreover the disjoint union of basis of each vertex group $\bigcup\limits_{v\in V(\cal{G})}\cal{B}_v$ is a subset of the basis of $\pi_{1}(\cal{G})$.

\begin{defn}\label{graphofgroupsiso}[Graph-of-groups Isomorphisms]

Let $\cal{G}_{1}$, $\cal{G}_{2}$ be two graphs of groups. 
Denote $\Gamma_{1}=\Gamma(\cal{G}_{1})$ and $\Gamma_{2}=\Gamma(\cal{G}_{2})$.
An isomorphism $H: \cal{G}_{1} \rightarrow \cal{G}_{2}$ 
is a tuple of the form
$$H=(H_{\Gamma}, (H_{v})_{v \in V(\Gamma_{1})}, (H_{e})_{e \in E(\Gamma_{1})}, (\delta(e))_{e \in E(\Gamma_{1})})$$
where
\begin{enumerate}
\item $H_{\Gamma}: \Gamma_{1} \rightarrow \Gamma_{2}$ is a graph isomorphism;
\item $H_{v}: G_{v} \rightarrow G_{H_{\Gamma}(v)}$ is a group isomorphism, for any $v \in V(\Gamma_{1})$;
\item $H_{e}=H_{\bar{e}}: G_{e} \rightarrow G_{H_{\Gamma}(e)}$ is a group isomorphism, for any $e \in E(\Gamma_{1})$;
\item for every $e \in E(\Gamma_{1})$, the {\it correction term} $\delta(e) \in G_{\tau(H_{\Gamma(e)})}$ is an element such that
    $$H_{\tau(e)}f_{e}=ad_{\delta(e)}f_{H_{\Gamma}(e)}H_{e}.$$
\end{enumerate}
\end{defn}

\begin{rem}
A graph-of-groups isomorphism 
$H: \cal{G}_{1} \rightarrow \cal{G}_{2}$ 
induces an isomorphism 
$H_{*}: \Pi(\cal{G}_{1}) \rightarrow \Pi(\cal{G}_{2})$ 
defined on the generators by:
\begin{enumerate}
\item[] $H_{*}(g) = H_{v}(g)$, for $g \in G_{v}$, $v \in V(\Gamma_{1})$;
\item[] $H_{*}(t_{e}) = \delta(\bar{e}) t_{H_{\Gamma}(e)} \delta(e)^{-1}$, for $e \in E(\Gamma_{1})$.
\end{enumerate}

It's easy to verify by computation that $H_{*}$ preserves the relations $t_{e}t_{\bar{e}}=1$ for any $e \in E(\cal{G})$ and $f_{\bar{e}}(g)=t_{e}f_{e}(g)t_{e}^{-1}$, for any $e\in E(\cal{G})$ and $g \in G_{e}$.

Furthermore, the restriction of $H_{*}$ 
to
$\pi_{1}(\cal{G}_{1},v)$, where $v \in V(\Gamma_{1})$, is also an isomorphism, denoted by $H_{* v}: \pi_{1}(\cal{G}_{1},v) \rightarrow \pi_{1}(\cal{G}_{2}, H_{\Gamma}(v))$.
\end{rem}

\begin{rem}
As in \cite{CL02}, we define the {\it outer isomorphism} induced by a group isomorphism $f: G_{1} \rightarrow G_{2}$ as the equivalence class 
$$\hat{f}=\{ad_{g}f : G_{1} \rightarrow G_{2} \mid g \in G_{2}\}.$$
Hence $H_{*v}$ induces an outer isomorphism $\hat{H}_{*v}: \pi_{1}(\cal{G}_{1},v) \rightarrow \pi_{1}(\cal{G},H_{\Gamma}(v))$.

Observe that when choosing a different vertex $v_{1}$ as basepoint, we may choose a word $W \in \Pi(\cal{G}_{1})$ 
with underlying 
path from $v_{1}$ to $v$ 
to obtain the following commutative diagram:
\[
\begin{CD}
\pi_{1}(\cal{G}_{1},v) @>H_{*v}>> \pi_{1}(\cal{G}_{2}, H_{\Gamma}(v)) \\
@Vad_{W} VV @Vad_{H_{*}(W)} VV \\
\pi_{1}(\cal{G}_{1},v_{1})@>>H_{*v_{1}}> \pi_{1}(\cal{G}_{2}, H_{\Gamma}(v_{1}))
\end{CD}
\]

By Lemma 2.2 and Lemma 3.10 in \cite{CL02}, $\hat{H}_{*v}$ determines an outer isomorphism $\hat{H}_{*v_{1}}: \pi_{1}(\cal{G}_{1}, v_{1}) \rightarrow \pi_{1}(\cal{G}_{2}, H_{\Gamma}(v_{1}))$. 

In this sense, we observe that $H: \cal{G}_{1} \rightarrow \cal{G}_{2}$ induces an outer isomorphism $\hat{H}: \pi_{1}(\cal{G}_{1}) \rightarrow \pi_{1}(\cal{G}_{2})$ which doesn't depend on the choice of basepoint.
\end{rem}

\subsection{H-conjugation}
\label{Hconj}

${}^{}$

We recall in this subsection some basic definitions and properties about the notion of $H$-conjugation. Contrary to the previous subsection, which only contained standard definitions and notation, the content of this subsection has been defined in \cite{KY01} and to our knowledge didn't exist previously.
 
\begin{defn}
For a graph-of-groups automorphism $H: \cal{G} \rightarrow \cal{G}$
two reduced words $w_{1}, w_{2} \in \Pi(\cal{G})$ are said to be 
{\it {H-conjugate}} to each other 
if there exists a reduced word $w\in \Pi(\cal{G})$ 
such that $w_{1}=w w_{2} H_{*}(w)^{-1}$.
\end{defn}

It's easy to show that $H$-conjugation is a well-defined equivalence relation on $\Pi(\cal{G})$.

Denote by $[w]_{H}$ the set which consists of all elements in $\Pi(\cal{G})$ that are $H$-conjugate to $w$. 
We call $[w]_{H}$ the {\it H-conjugacy class} of $w$.

Recall that the {\it {path length}} of a word $w \in \Pi(\cal{G})$ equals to the number of edges the path underlying $w$ crosses. 
We denote the path length of $w$ by 
$\| w \|_{\cal{G}}$.

\begin{defn}
A reduced word $w \in \Pi(\cal{G})$ is said to be {\it H-minimal} if it has the shortest path length among its $H$-conjugates.
More specifically, 
if $w$ is $H$-minimal, 
then for every $w_{0} \in \Pi(\cal{G})$, $\| w_{0} w H_{*}(w_{0})^{-1}\|_{\cal{G}} \geq \| w \|_{\cal{G}}$. 
\end{defn}

Since $\| w \|_{\cal{G}}$ is a natural number, 
one has that every reduced word $w \in \Pi(\cal{G})$ has a $H$-conjugate which is $H$-minimal.

Therefore there exists a well defined {\it H-length}:
\begin{align*}
\| w \|_{\cal{G},H}=\min \{\| g w H_{*}(g)^{-1} \|_{\cal{G}} \mid g \in \Pi(\cal{G})\}.
\end{align*} 

Moreover, $\| w \|_{\cal{G},H}=\| w \|_{\cal{G}}$ if and only if $w$ is H-reduced. 

\begin{defn}
\label{Hreduced-defn}
A reduced word $w \in \Pi(\cal{G})$ is called {\it {H-zero}} if and only if its $H$-length equals to zero, i.e. $\| w \|_{\cal{G},H}=0$. 
\end{defn}

It also can be shown that $w\in \Pi(\cal{G})$ is $H$-minimal if and only if it is $H$-reduced, as defined below:

\begin{defn}\label{Hreduced}
Let $w \in \Pi(\cal{G})$ be a reduced word in the form of
$w=r_{0}t_{1}r_{1}...r_{q-1}t_{q}r_{q}$,
$w$ is said to be {\it $H$-reduced} if its cannot be shortened by the {\it elementary operation}
$w \mapsto w_{1} = (r_{0}t_{1})^{-1} w H_{*}(r_{0}t_{1})$, 
i.e. $\| w \|_{\cal{G}} = \| w_1 \|_{\cal{G}}$.
\end{defn}

\begin{rem}

If $w^{\prime}=r_{0}t_{1}r_{1}...r_{k-1}t_{k}r_{k} \in [w]_{H}$ is $H$-reduced, then
\begin{align*}
& w^{\prime}_{1}  = (r_{0}t_{1})^{-1} w^{\prime} H_{*}(r_{0}t_{1})\\
& \ \ \ \ \ \ \  ...\\
& w^{\prime}_{i}  = (r_{i-1}t_{i})^{-1} w^{\prime}_{i-1} H_{*}(r_{i-1}t_{i})\\
& \ \ \ \ \ \ \ ... \\
& w^{\prime}_{k}  = (r_{k-1}t_{k})^{-1} w^{\prime}_{k-1} H_{*}(r_{k-1}t_{k}) \\
\end{align*}
are also $H$-reduced.

Moreover, since $H_{*}$ preserves the path lengths of reduced words, we also have $H_{*}(w^{\prime})=w^{\prime -1} w^{\prime } H_{*}(w^{\prime})$, $H_{*}^{-1}(w^{\prime})=H_{*}^{-1}(w^{\prime}) w^{\prime } w^{\prime -1}$ are $H$-reduced.


Hence $\bigcup\limits_{i=-\infty}^{+\infty}\{ H_{*}^{i}(w^{\prime}), H_{*}^{i}(w^{\prime}_{1}),...,H_{*}^{i}(w^{\prime}_{k}) \}$ covers all path types of $H$-reduced words in $[w]_{H}$.

\end{rem}

\begin{lem}\label{Hconj1}
Given a graph-of-groups isomorphism $H: \cal{G} \rightarrow \cal{G}$ which acts trivially on the underlying graph $\Gamma=\Gamma(\cal{G})$. 
Choose an arbitary vertex $v_0 \in V(\Gamma)$ as basepoint. 
For every closed reduced word  $W \in \pi_{1}(\cal{G},v_0)$, 
there exist a vertex $v_1 \in V(\Gamma)$ and a reduced word $\gamma \in \Pi(\cal{G})$ which underlies a path from $v_0$ to $v_1$ such that $\gamma^{-1} W H(\gamma) \in \pi_{1}(\cal{G}, v_1)$ is $H$-reduced. 
\end{lem}

\begin{proof}
In general, 
for every reduced word $W \in \Pi(\cal{G})$,
there exists $\gamma \in \Pi(\cal{G})$ such that $\gamma^{-1} W H(\gamma)$ is $H$-reduced. 
In the case where $H$ acts trivially on the graph $\Gamma$, 
the reduced words $\gamma$ and $H(\gamma)$ underly exactly the same edge path.
Hence the word $\gamma^{-1} W H(\gamma)$ is a closed word issued at the ternimal vertex of $\gamma$. 

Moreover, it derives from Section~\ref{Hconj} that we can find such an $H$-reduced word with path type that is a subsequence of the path type of $W$ by applying the elementary operation defined in Definition~\ref{Hreduced}.
\end{proof}

\section{Dehn twists}

\subsection{Classical Dehn twist}

\begin{defn}\label{classical dehn twist} [Classical Dehn twist] 
A {\it Dehn twist} $D$ of a graph-of-groups $\cal{G}$ is a graph-of-groups automorphism $D: \cal G \to \cal G$ which satisfies
(where $\Gamma = \Gamma(\cal G)$ denotes as usually the underlying graph):

\begin{enumerate}
\item $D_{\Gamma} = id_{\Gamma}$;
\item $D_{v} = id_{G_{v}}$, for all $v \in V(\Gamma)$;
\item $D_{e} = id_{G_{e}}$, for all $e \in E(\Gamma)$;
\item for each $G_{e}$, there is an element $\gamma_{e} \in Z(G_{e})$ such that the correction term satisfies $\delta(e)=f_{e}(\gamma_{e})$, where $Z(G_{e})$ denotes the center of $G_{e}$.
\end{enumerate}

We denote a Dehn twist defined as above by $D=D(\cal{G},(\gamma_{e})_{e \in E(\cal{G})})$
\end{defn}

\begin{rem}[Twistor]
Given a Dehn twist $D=D(\cal{G},(\gamma_{e})_{e \in E(\cal{G})})$, 
we define the {\it twistor} of an edge $e\in E(\Gamma)$ by setting
$z_{e}=\gamma_{\bar{e}}\gamma_{e}^{-1}$.
Then for any edge $e$ 
we have 
$z_{e} \in Z(G_{e})$ and $z_{\bar{e}}=\gamma_{e} \gamma_{\bar{e}}^{-1}=z^{-1}_{e}$.

\end{rem}

\begin{rem}
\label{rem-twistor}
The induced automorphism $D_{*} :\Pi(\cal{G}) \rightarrow \Pi(\cal{G})$ 
is defined on generators as
follows:
\begin{enumerate}
\item[] $D_{*}(g)=g$, for $g \in G_{v}$, $v\in V(\Gamma)$;
\item[] $D_{*}(t_{e}) = t_{e}f_{e}(z_{e})$, for every $e \in E(\Gamma)$.
\end{enumerate}

In particular, the induced automorphism on the fundamental group, 
$D_{*v}: \pi_{1}(\cal{G}, v) \rightarrow \pi_{1}(\cal{G},v)$ where $v \in V(\Gamma)$, 
is called a {\it Dehn Twist automorphism}.
\end{rem}

\begin{defn}\label{Dehntwist}
In general, 
a group automorphism $\phi : G \rightarrow G$ is said to be a {\it Dehn twist automorphism} 
if it is represented by a graph-of-groups Dehn twist. 
More precisely, 
there exists a graph-of-groups $\cal G$, 
a vertex $v$ of $\Gamma(\cal G)$,
a Dehn twist $D: \cal{G} \rightarrow \cal{G}$, 
and an isomorphsim $\theta: G \rightarrow \pi_{1}(\cal{G},v)$ such that $\phi=\theta^{-1} \circ D_{*v} \circ \theta$

In this case the induced outer automorphism $\hat{\phi}: G \rightarrow G$ is called a {\it Dehn twist outer automorphism}.
\end{defn} 

\begin{rem}
\label{some-D-lifts}
The reader may notice the following subtlety in the above definitions: 

Because of the role of the base point $v$ in Definition \ref{Dehntwist}, it may well occur that two automorphisms $\phi_1$ and $\phi_2$ of a group $G$ define the same outer automorphism $\hat \phi_1 = \hat \phi_2$ which is a Dehn twist outer automorphism, but only $\phi_1$ is a Dehn twist automorphism, while $\phi_2$ isn't.
\end{rem}

\begin{prop}[Proposition 5.4 \cite{CL02}]
Suppose
$\cal{G}$ is a graph-of-groups which satisfies that 
for every edge $e$ there is an element $r_{e} \in G_{\tau(e)}$ with 
$$f_{e}(G_{e}) \cap r_{e}f_{e}(G_{e})r_{e}^{-1}=\{1\}.$$ 

Then two Dehn twists 
$D=(\cal{G},(\gamma_{e})_{e\in E(\cal{G})})$, 
$D^{\prime}=(\cal{G},(\gamma^{\prime}_{e})_{e \in E(\cal{G})})$ 
determine the same outer automorphism of $\pi_{1}(\cal{G})$ 
if and only if $z_{e}=z_{e}^{\prime}$ for all $e \in E(\Gamma)$.
\end{prop}

This proposition shows that in many situations a Dehn twist on a given graph-of-groups is uniquely determined by its twistors.
Thus sometimes we may define a Dehn twist by its twistors 
$(z_{e})_{e\in E(\Gamma)}$ 
(for each $e\in E(\Gamma)$, $z_{e} \in Z(G_{e})$ and $z_{\bar{e}}=z_{e}^{-1}$). 
In this case, 
we may conversely define:
\begin{equation*}
\gamma_{e}=\left\{
\begin{aligned}
& z_{e}^{-1}, & e\in E^{+}(\Gamma) \\
& 1,  & e\in E^{-}(\Gamma).
\end{aligned}
\right.
\end{equation*}

\subsection{General and partial Dehn twists}
${}^{}$

As discussed in \cite{KY01}, 
we can define a Dehn twist in a slightly more general context 
by replacing the last condition of 
Definition~\ref{classical dehn twist} by the following:
\begin{enumerate}
\item [(4*)] the correction term $\delta(e) \in C(f_{e}(G_{e}))$, where $C(f_{e}(G_{e}))$ denotes the centeralizer of $f_{e}(G_{e})$ in $G_{\tau(e)}$, for all $e \in E(\Gamma)$. 
\end{enumerate}

We call a graph-of-groups automorphism which satisfies conditions $(1)-(3)$ in Definition~\ref{classical dehn twist} and  $(4*)$ a {\it general Dehn twist}.

It's shown in \cite{KY01} that Dehn twists defined in either, the classical or the  general version, are equivalent in the sense that: 
(i) every classical Dehn twist is a general Dehn twist; 
(ii) every general Dehn twist has a naturally corresponding classical Dehn twist which induces same outer automorphism.

On other hand,
if $\cal{G}$ is a graph-of-groups with trival edge groups, 
and $H: \cal{G} \rightarrow \cal{G}$ is an automorphism which acts trivially on the graph and vertex groups, then
it follows immediately from the above definitions that $H$ induces Dehn twist automorphisms, for any choice of the family of correction terms.

\begin{defn}
\label{partial-Dehn-twists}
(a)
A {\it partial Dehn twist relative to a subset of vertices} 
$\cal{V}=\{v_1, v_2, ..., v_m\}$ of $\cal G$ is  
a graph-of-groups automorphism $H: \cal{G} \rightarrow \cal{G}$ such that
\begin{enumerate}
\item[(1)] for every $e \in E(\cal{G})$ with $\tau(e) \in \cal{V}$, 
the edge group $G_e$ is trivial;
\item[(2)] $H$ satisfies all conditions of a general Dehn twist 
except at $v_i \in \cal{V}$. 
That is to say, 
any of vertex group automorphism $H_{v_i}$ with
$v_i \in \cal{V}$ may not be trivial.
\end{enumerate} 

\smallskip
\noindent
(b)
More specifically, a {\it partial Dehn twist relative to a family of local Dehn twists} is a partial Dehn twist relative to a subset of vertices $\cal V$ of $\cal G$, and at any vertex $v \in \cal V$ the vertex group automorphism $H_v: G_v \to G_v$ is a Dehn twist automorphism.
\end{defn}

\begin{rem}
\label{partial-D-rel-D}
If some automorphism $\phi: G \to G$ is represented by a partial Dehn twist $H: \cal G \to \cal G$ relative to a family of vertices $v_i \in \cal V \subset V(\Gamma(\cal G))$, where every vertex group automorphism $H_{v_i}: G_{v_i} \to G_{v_i}$ induces an outer Dehn twist automorphism $\hat H_{v_i}$, 
then $\phi$ is a partial Dehn twist relative to a family of local Dehn twists. 

This can be seen through replacing $H$ by $J \circ H$, where the graph-of-groups automorphism $J: \cal G \to \cal G$ is the identity on all edge groups and on all vertex groups for vertices outside $\cal V$, and an inner automorphism on all $G_{v_i}$ if $v_i \in \cal V$. Here the correction terms $\delta^J_e$ for edges $e$ with terminal vertex $\tau(e) \notin \cal V$ are trivial, while for edges e with $\tau(e) \in \cal V$ they are properly chosen to ``undo'' the inner automorphism $J_{\tau(e)}$ on $G_{\tau(e)}$,
so that for any $v \notin \cal V$ the induced automorphism $J_{*v}: \pi_1(\cal G, v) \to \pi_1(\cal G, v)$ is the identity map (and thus for any $v \in \cal V$ the automorphism $J_{*v}$ is an inner automorphism).
See Section~2.4 in \cite{KY01} for more details.
\end{rem}

\subsection{Efficient Dehn Twist}
\label{efficient-subsection}

${}^{}$

Unless otherwise stated, 
in this subsection we always assume 
$D: \cal{G} \rightarrow \cal{G}$ is a Dehn twist defined in the classical meaning.
We write $D = D(\cal{G}, (z_{e})_{e \in E(\cal{G})})$.


Two edges $e_{1}$ and $e_{2}$ with common terminal vertex $v$ are called

\begin{enumerate}

\item[$\diamond$]
{\it {positively bonded}}, if there exist $n_1,n_2 \geq 1$ such that $f_{e_{1}}(z_{e_{1}}^{n_1})$ and  $f_{e_{2}}(z_{e_{2}}^{n_2})$ are conjugate in $G_{v}$.

\item[$\diamond$]
{\it {negative bonded}}, if there exist $n_1 \geq 1$, $n_2 \leq 1$ such that $f_{e_{1}}(z_{e_{1}}^{n_1})$ and  $f_{e_{2}}(z_{e_{2}}^{n_2})$ are conjugate in $G_{v}$.

\end{enumerate}

For the rest of this subsection, we always assume for a graph-of-groups $\cal{G}$ its fundamental group $\pi_{1}(\cal{G})$ is free and of finite rank $n \geq 2$.
This implies, by definition of a classical Dehn twist, that any edge $e$ with non-trivial twistor $z_e$ has edge group $G_e \cong \Z$.

\begin{defn}[Efficient Dehn twist \cite{CL02}]
\label{efficient-D-twists-defn}
A Dehn twist $D=D(\cal{G},(z_{e})_{e \in E(\cal{G})})$ is said to be {\it efficient} if :

The graph-of-groups $\cal{G}$ satisfies
\begin{enumerate}
\item $\cal{G}$ is {\it minimal}: if $v = \tau(e)$ is a valence-one vertex, then the edge homomorphism $f_{e}: G_{e}\rightarrow G_{v}$ is not surjective.
\item There is no {\it invisible vertex}: there is no valence-two vertex $v=\tau(e_{1})=\tau(e_{2})$ $(e_{1} \neq e_{2})$ such that both edge maps $f_{e_{i}}: G_{e_{i}} \rightarrow G_{v}$ $(i=1,2)$ are surjective. 
\item No {\it proper power}: if $r^{p} \in f_{e}(G_{e})$ $(p \neq 0)$ then $r \in f_{e}(G_{e})$, for all $e \in E(\Gamma)$.
\end{enumerate}

And together with the collection of twistors $(z_{e})_{e \in E(\Gamma)}$, it also satisfies:
\begin{enumerate}
\item[(4)] No {\it unused edge}: for every $e \in E(\Gamma)$, the twistor $z_{e} \neq 1$ (or equivalently $\gamma_{e} \neq \gamma_{\bar{e}}$).
\item[(5)] If $v= \tau(e_{1})=\tau(e_{2})$, then $e_{1}$ and $e_{2}$ are not positively bonded.
\end{enumerate}
\end{defn}

The following has been shown in \cite{CL02}:

\begin{prop}
\label{vertex-rank}
For every vertex $v \in V(\cal G)$ of an efficient Dehn twist $D: \cal G \to \cal G$ the group $G_v$ has rank
$$\rm{rk} (G_v) \geq 2 \, .$$
\qed
\end{prop}

\begin{prop}
\label{fixed-elements}
Let $D: \cal G \to \cal G$ be an efficient Dehn twist, and let $v \in V(\cal G)$ be any vertex.

\smallskip
\noindent
(a)
A conjugacy class $[w]$ of $\pi_1 (\cal G, v)$ is fixed by $\hat D_{*v}$ if and only if $w$ has 
$H$-length
$$\| w \| _{\cal{G},H}= 0 \, .$$

\smallskip
\noindent
(b)
An element $w \in \pi_1 (\cal G, v)$ is fixed by $D_{*v}$ if and only if $w \in G_v$.
\qed
\end{prop}

It's shown in \cite{CL02} that every Dehn twist (classical or general) can be transformed algorithmically into an efficient Dehn twist, and furthermore, the latter is essentially unique:

\begin{thm}
\label{Efficient}
(1)
For every Dehn twist $D=D(\cal{G},(z_{e})_{e \in E(\Gamma)})$, there exists an efficient Dehn twist $D^{\prime}=D(\cal{G^{\prime}},(z_{e})_{e \in E(\Gamma^{\prime})})$ and an isomorphism between fundamental groups $\rho: \pi_{1}(\cal{G}, w) \rightarrow \pi_{1}(\cal{G}^{\prime}, w^{\prime})$, where $w\in V(\Gamma)$, $w^{\prime}\in V(\Gamma^{\prime})$ are properly chosen vertices, such that $\hat{D^{\prime}}\hat{\rho}=\hat{\rho}\hat{D}$.

\smallskip
\noindent
(2)
If $D'^{\prime}=D(\cal{G'^{\prime}},(z_{e})_{e \in E(\Gamma'^{\prime})})$ is a second such efficient Dehn twist, with respect to a analogous fundamental group isomorphism $\rho_0$, then there exists a graph-of-groups isomorphism $H: \cal G' \to \cal G''$ with $D'' = H D' H^{-1}$ and $\hat \rho_0 = \hat H \hat \rho$.

\qed
\end{thm}




We now return to the issue of $H$-conjugation as recalled in the last section, but with the specification that the graph-of-groups isomorphism $H: \cal G \to \cal G$ is a Dehn twist $D$, and we are interested in $H$-reduced or rather $D$-reduced elements as discussed in Definition \ref{Hreduced-defn} and Remark \ref{Hreduced}.

\begin{rem}[D-reduced v.s. cyclically reduced]
${}^{}$

\noindent
(a)
If the Dehn twist $D: \cal G \to \cal G$ is defined in the classical way,
then an element $w \in \Pi(\cal{G})$ is $D$-reduced if and only if it's cyclically reduced. 
In particular, if for some vertex $v_0 \in V(\cal G)$ an element $w \in \pi_1(\cal G, v_0) \subset \Pi(\cal G)$ is $D$-reduced, then $w$ is $D$-zero if and only if it satisfies $\|w\|_{\cal G} = 0$.

\smallskip
\noindent
(b)
However, for $w \in \Pi(\cal{G})$ which is not $D$-reduced (or equivalently, not cyclically reduced), its usual conjugates which are cyclically reduced and its $D$-conjugates which are $D$-reduced can be very different, in fact they may not even have the same path type. 

This shows that even in case of an efficient Dehn twist $D$ for 
non-D-reduced 
word in $\Pi(\cal G)$ being $D$-zero and having its cyclically reduced path length equals to zero are not equivalent.
\end{rem}

In view of the existence result for efficient Dehn twist representatives from Theorem \ref{Efficient} (1) we will always, when an outer Dehn twist automorphism $\hat \phi \in \Out(\FN)$ is given without specification of a Dehn twist representative, assume that it is represented by an efficient Dehn twist $D: \cal G \to \cal G$. Similarly, 
a Dehn twist automorphism $\phi \in \Out(\FN)$ without specification of a Dehn twist representative is always assumed to be represented by an efficient Dehn twist $D: \cal G \to \cal G$. 

Recall from Remark \ref{some-D-lifts} that the last convention may appear slightly restrictive, in particular when it comes to a partial Dehn twist $D : \cal G \to \cal G$ relative to a subset $\cal V$ of vertices $v_i$ of $\cal G$ for which the induced vertex group automorphism $D_{v_i}$ is known to induce an outer Dehn twist automorphism $\hat D_{v_i}$. However, if follows from
Remark \ref{partial-D-rel-D} that this restriction is immaterial; this class of automorphisms is precisely the same as the one given in Definition \ref{partial-Dehn-twists} (b), i.e. the class of partial Dehn twists relative to a family of local Dehn twists.

In view of the uniqueness of $D$ affirmed by part (2) of Theorem \ref{Efficient}, 
the following notion is well defined:

\begin{defn}
\label{phi-zero}
Let $\phi \in \Aut(\FN)$ be a Dehn twist automorphism. Then any element $w \in \FN$ is called $\phi$-reduced (or $\phi$-zero) if it is $D$-reduced (or $D$-zero) with respect to some efficient Dehn twist representative of $\phi$.
\end{defn}

\section{Cancellation Results}

\subsection{Growth type}

${}^{}$

We first recall some standard notation and well know elementary facts:

\begin{defn}
Let $G$ be a finitely generated group, and let $X=\{x_{1}, x_{2},...,x_{n}\}$ denote its generating set. The {\it length function} with respect to the generating set $X$ is defined by setting for any $g \in G$:
$$|g|_{X} := \min \{k \geq 0 \mid g=x_{i_{1}}^{\epsilon_{1}}x_{i_{2}}^{\epsilon_{2}}...x_{i_{k}}^{\epsilon_{k}}, i_{j} \in \{1,...,n\}, \epsilon_{j} \in \{\pm 1\}\}.$$

The {\it cyclic length} of $g \in G$ is defined by
$$\|g\|_{X}:= \min \{|hgh^{-1}|_{X} \mid h\in G\}.$$
\end{defn}

\begin{rem}
\
\begin{enumerate}
\item 
For any $g \in G$ we have $|g|_{X} \geq 0$, and $|g|_{X}=0$ holds if and only of $g=1$.

\item For any $g \in G$
the cyclic length $\|g\|_{X}$ is the minimum of all lengths of elements in the conjugacy class $[g]$. 
The elements $h \in G$ and $hgh^{-1}$ such that $|g|_{X}=|hgh^{-1}|_{X}$ may not be unique.

\item 
If $G = \FN$ and $X$ is a basis, we also have
$$\|g\|_{X}=|gg|_{X}-|g|_{X}.$$
Furthermore
$\|g\|_{X}=|g|_{X}$ if and only if $g \in G$ is cyclically reduced.

\item For any words $g, h \in F_n$ we always have $|gh| \leq |g|+|h|$.
\end{enumerate}
\end{rem}

\begin{rem}
For two sets of generators $X=\{x_{1}, x_{2},...,x_{n}\}$ and $X^{\prime}=\{x^{\prime}_{1}, x^{\prime}_{2},...,x^{\prime}_{n}\}$ of $G$, the length fonctions $| \cdot |_{X}$ and $|\cdot |_{X^{\prime}}$ are equivalent up to a constant. To be more precise, there exists a constant $C>0$ such that for all $g \in G$:
$$\frac{1}{C}|g|_{X} \leq |g|_{X^{\prime}} \leq C|g|_{X}.$$
\end{rem}

\begin{defn}
Let $\phi \in Aut(G)$ be an automorphism and $X$ be any generating system of $G$. For any element $g \in G$ we introduce a function $Gr(\phi,g)$ 
to trace the length of $g$ under the iteration of $\phi$: 
\begin{align*}
Gr(\phi,g)(n): & \,  \Bbb{N} \rightarrow \Bbb{N}\\
& n  \mapsto  |\phi^{n}(g)|_{X}
\end{align*}
Similarly, for the cyclic length we have:
\begin{align*}
Gr_{c}(\phi,g)(n): & \,\Bbb{N} \rightarrow \Bbb{N}\\
& n  \mapsto  \| \phi^{n}(g) \|_{X}
\end{align*}
\end{defn}

Notice that for $\phi\in Aut(G)$, $g_{1}, g_{2} \in [g]$, and $n \in \Bbb{N}$ one has:
$$Gr_{c}(\phi,g_{1})(n)=Gr_{c}(\phi,g_{2})(n)$$

Also, for $\hat \phi \in Out(G)$ and $\phi_{1}, \phi_{2} \in \hat{\phi}$ we obtain
$$Gr_{c}(\phi_{1},g)(n)=Gr_{c}(\phi_{2},g)(n)$$ 
for all $g \in G$, $n \in \Bbb{N}$. 
Thus it makes sense to consider the cyclic length of the conjugacy class $[g]$ (or equivalently $g \in [g]$) under the iteration of the outer automorphism $\hat{\phi} \in Out(G)$.

\begin{defn}[growth type]
\label{growth-type-defn}
(a)
We say that $g\in G$ {\it grows at most polynomially of degree $d$} under iteration of $\phi \in Aut(G)$ if $Gr(\phi,g)$ is bounded above by a polynomial of degree $d$. 
The conjugacy class $[g]$ {\it grows at most polynomially of degree $d$} under iteration of $\hat{\phi} \in Out(G)$ (or equivalently, of $\phi \in \hat{\phi}$) if $Gr_{c}(\hat{\phi},[g])$ (or $Gr_{c}(\phi,[g])$) is bounded above by a polynomial of degree $d$.

The automorphism $\phi \in Aut(G)$ {\em has at most polynomial growth of degree $d$} if any $g \in G$ grows at most polynomially of degree $d$. Similarly, the outer automorphism $\hat \phi \in Aut(G)$ {\em has at most polynomial growth of degree $d$} if any $[g] \subset G$ grows at most polynomially of degree $d$.

\smallskip
\noindent
(b)
Similarly we say that $g\in G$ (or [g]) {\it grows at least polynomially of degree d} under iteration of $\phi \in Aut(G)$ (or of $\hat{\phi} \in Out(G)$ respectively) if $Gr(\phi,g)$ (or $Gr_{c}(\hat{\phi},[g])$) is bounded below by a polynomial of degree d with positive leading coefficient.

The automorphism $\phi \in Aut(G)$ {\em has at least polynomial growth of degree $d$} if some $g \in G$ grows at least polynomially of degree $d$. Similarly, the outer automorphism $\hat \phi \in Aut(G)$ {\em has at least polynomial growth of degree $d$} if some $[g] \subset G$ grows at least polynomially of degree $d$.

\smallskip
\noindent
(c)
If $g$ (or $[g]$) grows both at most and at least polynomially of degree $d$, then we say that it grows polynomially of degree $d$.

\smallskip
\noindent
(d)
An automorphism $\phi \in Aut(G)$ (or an outer automorphism $\hat{\phi} \in Out(G)$) {\it grows polynomially of degree d} if every $g \in G$ 
(or every $[g] \subset G$)
grows at most polynomially of degree $d$ and in particular there exists an element $g_{0} \in G$ 
(or a conjugacy class $[g] \subset G$)
that grows polynomially of degree $d$.
\end{defn}

\begin{rem}
Because the length functions with respect to different generating systems are equivalent up to a constant, the definitions about growth types given in Definition \ref{growth-type-defn} are independent of the chosen generating system of $G$.
\end{rem}

\begin{defn}
Let $(w_k)_{k=1}^{+\infty}$ a family of elements in $G$, we sometimes say that the sequence $w_k$ grows {\it at least polynomially of degree d} if there exists constant $C_1>0$ such that $C_1 k^d \leq |w_k|$. Similarly, we say that $w_k$ grows {\it at most polynomially of degree d} if there exists $C_2 >0$ such that $|w_k| \leq C_2 k^d$ and that $w_k$ grows {\it polynomially of degree d} if one can find $C_1, C_2>0$ such that $C_1 k^d \leq |w_k| \leq C_2 k^d$.
\end{defn}

\subsection{Cancellation and iterated products}
 
${}^{}$

Let $F_n$ be a free group and denote by $\cal{A}$ a fixed basis of $F_n$.
As before, we denote the combinatorial length (with respect to $\cal{A}$) of an element $W$ 
by $|W|=|W|_{\cal{A}}$, 
and the cyclically reduced length of $W$ 
by $\| W \| = \| W \|_{\cal A}$.

\begin{lem}
\label{cancellation-blocker}
Let $F_n$ be a free group, let $V \subset \FN$ be a subgroup of rank $n \geq 2$, and let $(w_i)_{i \in \N}$ be any infinite family of elements $w_i \in F_n$. Then for any basis $\cal A$ of $F_n$ there exists an element $v \in V$ and a constant $C \geq 0$ such that for infinitely many indices $i \in \N$ the cancellation in the product $|w_i v w_i^{-1} |_{\cal A}$ is bounded, i.e. one has:
$$|w_i v w_i^{-1} |_{\cal A} \geq |w_i|_{\cal A} + |v|_{\cal A} + |w_i^{-1}|_{\cal A} - C = 2|w_i|_{\cal A} + |v|_{\cal A} - C$$
\end{lem}

\begin{proof}
Pick elements $v_1$ and $v_2$ in $V$ which generate a subgroup of rank 2. Consider the products $w_i v_1^m w_i^{-1}$, 
for increasing integers $m$. We observe that one of the following must hold:
\begin{enumerate}
\item
For some $m \in \N$ the cancellation in $w_i v_1^m w_i^{-1}$ is bounded uniformly with respect to all $i \in \N$.
\item
For any $m \in \N$ there is an index $j(m) \in \N$ such that $w_{j(m)}$ has the suffix $v_1^{-m}$.
\item
For any $m \in \N$ there is an index $j'(m) \in \N$ such that $w_{j'(m)}^{-1}$ has the prefix $v_1^{-m}$, or equivalently, $w_{j'(m)}$ has the suffix $v_1^{m}$.
\end{enumerate}
In the case of alternative (1), the proof is completed. In case of (2), we replace the family of all $w_i$ by the subfamily of all $w_j$ with $j = j(m)$ for any $m \in \N$. It follows from the statement (2) that this subfamily is infinite. In case (3) we do the same, but with $j = j'(m)$.

We now repeat the above trichotomy, with $w_i$ replaced by $w_j$, and with $v_1$ replaced by $v_2$. We observe that in this second trichotomy the alternatives (2) and (3) lead to elements $w_j$ with suffix that is simultaneously an arbitrary large positive or negative power of $v_1$ and of $v_2$. But this is impossible, by our assumption that $v_1$ and $v_2$ generate a subgroup of rank 2. Thus alternative (1) must hold for the second trichotomy, which proves our claim.
\end{proof}

\subsection{Cancellation in long products}
${}^{}$

\begin{defn} 
We say $U$ and $V$ admit a common root if there exist an element $R \in F_n$, $R \neq 1$, such that $U=R^{m_1}$, $V=R^{m_2}$ for some suitable $m_1, m_2 \in \Bbb{N}$, 
$R$ is called a {\it common root} of $U$ and $V$.
\end{defn}

We recall the following well known fact:

\begin{prop}[\cite{Lyndon}]

For any elements $U, V \in F_n$, 
there is an algorithm which decides whether they admit a common root.
\end{prop}

The following is well-known too: 

\begin{lem}\label{lemma1}
For two elements $U, V \in F_n$ 
we have $U^{n_1} \neq V^{n_2}$ for any $n_1, n_2 \geq 1$ if and only if $U,V$ don't admit a common root.
\end{lem}

\begin{proof}
On one hand, if $U,V$ admit a common root $R$ such that $U=R^{m_1}$, $V=R^{m_2}$ for some $m_1, m_2 \geq 1$, 
then we have $U^{m_2}=V^{m_1}$.

On the other hand, 
if there exist $n_1, n_2 \geq 1$ such that $U^{n_1}=V^{n_2}$,
by comparing the subfixes and prefixes of $U$ and $V$, we can find a common root $R \in F_n$.
\end{proof}

\begin{lem}\label{lemma2}
If $U^{-1}, V \in F_n$ don't admit a common root, 
then the cancellation of the products $U^{n_1} V^{n_2}$, 
for any $n_1, n_2 \in \Bbb{N}$, 
is uniformly bounded. 
As consequence, there exists a constant and $K_0=K(U^{-1},V)$ such that for any $n_1, n_2 \in \Bbb{N}$, 
we have:
$$|U^{n_1}V^{n_2}| \geq n_1 \| U \| + n_2 \|V \| + K_0 $$ 
\end{lem}

\begin{proof}
If no constant $K_0$ as postulated exists,  then by comparing the subfixes and prefixes of $U$ and $V$ we can find a common root $R \in F_n$ for $U^{-1}$ and $V$, which contradicts our hypothesis.

Therefore by definition we can find $B_1 \geq 0$ such that 
$$|U^{n_1}|+|V^{n_2}|-|U^{n_1}V^{n_2}| \leq B_1.$$
Hence, by taking $K_0 = -B_0$ we have
$$|U^{n_1}V^{n_2}| \geq |U^{n_1}| + |V^{n_2}| +K_0 \geq n_1 \| U \| + n_2 \|V\| + K_0 $$
for any $n_1, n_2 \geq 0$.

\end{proof}

\begin{rem}
Furthermore,
for $U^{-1}, V \in F_n$ which don't admit a common root, 
we have at the same time that the cancellation of the products $V^{n_2}U^{n_1}$, 
for any $n_1, n_2 \in \Bbb{N}$, 
is uniformly bounded by some constant $B_2 \geq 0$.
Taking $K=-B_1 - B_2$, we have immediately 
$$\|U^{n_1}V^{n_2}\| \geq n_1 \| U \| + n_2 \|V\| + K .$$ 
\end{rem}

\begin{lem}\label{lemma3}
Let $X, b, Y \in F_n$ be elements such that 
$X^{-m_1} \neq b Y^{m_2}b^{-1}$, for any $m_1, m_2 \geq 1$.
Then there exists a constant $K=K(X,b,Y)$ 
such that for any $n_1, n_2 \geq 0$ we have:
$$| X^{n_1}b Y^{n_2} | \geq n_{1} \| X \| + n_2 \| Y \| + K.$$

\end{lem}

\begin{proof} 
For any $n_1, n_2 \geq 0$, we may consider the word 
$$X^{n_1}b Y^{n_2}=X^{n_1} bY^{n_2}b^{-1} b = X^{n_1} (bYb^{-1})^{n_2} b.$$

Taking $U=X$, $V=b Y b^{-1}$, we know from Lemma~\ref{lemma1} that $U^{-1}, V$ don't admit a common root. Hence it follows from Lemma~\ref{lemma2} that there exists $K_0=K(U^{-1},V)$ such that 
$$|U^{n_1}V^{n_2}| \geq n_1 \| U \| + n_2 \|V\| + K_0 .$$ 

Let $K=K_0 - |b|$, we have immediately the inequality
\begin{align*} 
|X^{n_1}b Y^{n_2}| & =|U^{n_1}V^{n_2}b| \\
& \geq |U^{n_1}V^{n_2}|-|b| \geq n_1 \| U \| + n_2 \|V \| + K=n_1 \| X \| + n_2 \|Y\| + K.
\end{align*}

\end{proof}

\begin{rem}

Please notice that in 
the situation considered in Lemma \ref{lemma3} 
we may {\em not} have the following inequality:
$$\| X^{n_1}b Y^{n_2} \| \geq n_{1} \| X \| + n_2 \| Y \| + K$$

Counter-example: Consider $F_2=\langle a, b \rangle$ and let $X=a^{-1}$, $Y=a$.

\end{rem}

\begin{lem}\label{remark1}
Let $X, b, Y \in F_n$ be elements such that 
$X^{-m_1} \neq b Y^{m_2}b^{-1}$, for any $m_1, m_2 \geq 1$.
Then there exist cyclically reduced conjugates $\tilde{X}, \tilde{Y}$ of $X,Y$ and $n_0 \in \Bbb{N}$ 
such that for all $n_1, n_2 \geq n_0$,
in the reduced product of $X^{n_1}b Y^{n_2}$
neither $\tilde{X}$ nor $\tilde{Y}$ is completely cancelled.
\end{lem}

\begin{proof}
Consider the cyclically reduced conjugates 
$\tilde{X}=w_1 X w_1^{-1}$, $\tilde{Y}= w_2 Y w_2^{-1}$ for $X$, $Y$, 
where $w_1, w_2 \in F_n$.
We may then write the word
$$X^{n_1}b Y^{n_2}= w_1^{-1}\tilde{X}^{n_1} w_1 b w_2^{-1} \tilde{Y}^{n_2} w_2^{-1}.$$

We derive from the uniformly bounded property that there exists $n_0 \in \Bbb{N}$ such that when $n_1, n_2 \geq n_0$,
in the reduced product of $X^{n_1}b Y^{n_2}= w_1^{-1}\tilde{X}^{n_1} w_1 b w_2^{-1} \tilde{Y}^{n_2} w_2^{-1}$, 
neither $\tilde{X}$ nor $\tilde{Y}$ is completely cancelled. 

More concretely, since we always have the inequality
$$|X^{n_1}b Y^{n_2}| \leq n_1 \|X\|+ n_2 \|Y\|+2 |w_1|+2|w_2|+|b|,$$
it follows from Lemma~\ref{lemma3} that the cancellation in the products $X^{n_1}bY^{n_2}$ is uniformly bounded by $B=2 |w_1|+2|w_2|+|b|-K$, where $K=K(X,b,Y)$ is the constant obtained in Lemma~\ref{lemma3}.
Then we may choose $n_0 \in \Bbb{N}$ that satisfies $\|X^{n_0}\|=n_0 \|X\|,\|Y^{n_0}\|=n_0 \|X\| \geq B$.

\end{proof}

\subsection{Main cancellation result}

${}^{}$

Let now $\cal{F}$ be a set which consists of finitely many triplets $(X_i, b_j, X_k)$, 
where $X_i, b_j, X_k \in F_n$ are elements which satisfy 

\begin{enumerate}
\item $\| X_i \| >0$ and $\| X_k \| >0$, and 
\item 
$X_{i}^{-m_1} \neq b_{j}X_{k}^{m_2}b_{j}^{-1}$, for any $m_1 \geq 1$, $m_2 \geq 1$.
\end{enumerate}
 
We consider below words 
$$w=w(n_1,n_2,...n_q)=c_0 y_{1}^{n_1}c_1 y_{2}^{n_2} c_2 \ldots c_{q-1} y_{q}^{n_q} c_q \in F_n$$ 
which have the property 
$(y_i, c_i, y_{i+1}) \in \cal{F}$, for $1 \leq i \leq q$.

\medskip

We then derive the following proposition:

\begin{prop}\label{prop1}
There exist constants $N_0 \geq 0$ and $K_0$ 
such that 
for any 
$w=w(n_1,n_2,...n_q) \in F_n$ as above
and any
$n_i \geq N_0$  (with $1 \leq i \leq q$),
we have
$$|w| \geq \sum\limits_{i=1}^{q} n_i \| y_i \| + (q-1) K_0.$$
\end{prop}

\begin{proof}
It follows directly from Lemma~\ref{lemma3} and Remark~\ref{remark1} that for each triplet $(y_i, c_i, y_{i+1})$, 
$1 \leq i \leq q-1$,
there exist constants $K_i=K(y_i, c_i, y_{i+1})$ and $N_i \geq 0$ such that :
$$|y_i^{n_i} c_i y_{i+1}^{n_{i+1}}| \geq n_i \|y_i\| + n_{i+1} \|y_{i+1}\| + K_i, $$
Moreover, if $n_i, n_{i+1} \geq N_i$
neither of the cyclically reduced conjugates 
$\tilde{y}_i= w_i y_i w_{i}^{-1}$, 
$\tilde{y}_i+1= w_{i+1} y_{i+1} w_{i+1}^{-1}$
is completely cancelled in the reduced product
$$y_i^{n_i} c_i y_{i+1}^{n_{i+1}}= w_i^{-1} y_i^{n_i} w_{i} c_i w_{i+1}^{-1} y^{n_{i+1}}_{i+1} w_{i+1}.$$

We now prove the proposition by induction.
\begin{enumerate}
\item
The case for $q=1$ is trivial while the case for $q=2$ is shown in Lemma~\ref{lemma3}.
\item
Suppose the inequality holds for $q=s$. 
In other words:
we can find constants $N \geq N_{s-1}$ and $K$ such that for  $n_i \geq N$ 
$$|c_0 y_1^{n_1} c_1 y_2^{n_2} c_2 ...c_{s-1} y_s^{n_s} c_s| \geq \sum\limits_{i=1}^{s} n_i \|y_i\| + (s-1)K$$
and $\tilde{y}_s^{n_s}$ is not completely cancelled in the reduced procedure.

%

\end{enumerate}

In particular, 
given that in each inductive step
the constants
$N$ and 
$K'$ depends only on the triplets 
$(y_i, c_i, y_{i+1})$, for $1 \leq i \leq 
q-1$
one can in fact deduce the final cancellation bound $(q-1)K_0$ based on just the family $\cal{F}$.  
In other words, the cancellation bound $K_0$ doesn't depend on the exponents $n_i$'s, 
once they are bigger than $N_0 := N$.

\end{proof}

\begin{rem}
In addition if $(y_q, c_qc_0, y_1) \in \cal{F}$,
similarly to what is done in the last proof, we may apply the same technique to the triplet $(y_q, c_qc_0, y_1)$ and obtain the following estimate for cyclical length of $w$ 
(again assuming $n_i \geq N_0$ for all exponents $n_i$): 
$$\| w \| \geq \sum\limits_{i=1}^{q} n_i \| y_i \| + q K_0$$
\end{rem}

\subsection{Cancellation bounds for $\cal T$-products}

${}^{}$

Let $\cal{T} \subset F_n \smallsetminus \{1\}$ be a finite set. We say that a product 
\begin{equation}
\label{T-word}
W = W(w_i, y_i) = w_0 y_1 w_1 y_2 w_2 \ldots w_{q-1} y_q w_q
\end{equation}
is a {\em $\cal{T}$-word} of {\em $\cal{T}$-length} 
$$|W|_\cal{T} = q\, ,$$

if $w_i \in F_n$ and $y_i \in \cal{T}$ for all indices $i$, 
and we say that the product $W$ is {\em $\cal{T}$-reduced} if 
$y_i^{-m} \neq w_i y_{i+1}^{m'} w_i^{-1}$ for any integers $m, m' \geq 0$.

For any $\cal{T}$-word $W$ as in (\ref{T-word}) and any multi-exponent 
\begin{equation}
\label{multi-exponent}
[n] = (n_1, n_2, \ldots n_q)  \in \N^q
\end{equation} 
we denote by $W^{[n]}$ the word
$$W^{[n]} = w_0 y_1^{n_1} w_1 y_2^{n_2}n w_2 \ldots w_{q-1} y_q^{n_q} w_q \, .$$
For any $n_0 \in \Z$ we write 
$$[n] \geq n_0$$
if $n_i \geq n_0$ holds for all components $n_i$ of $[n]$.

\begin{prop}
\label{T-word-length}
(a)
For any reduced $\cal{T}$-word $W$ as in (\ref{T-word}) and any basis $\cal A$ of $F_n$ there exist constants $K = K(W) \geq 0$ 
and $n_0 \in \Z$ 
such that for any multi-exponent $[n] \in \N^q$ as in (\ref{multi-exponent}), 
with $[n] \geq n_0$, 
the $\cal A$-lengths 
satisfy
$$|W^{[n]}|_{\cal A} \geq   \sum_{k = 1}^{q} n_k \|y_k\|_{\cal A} - K$$

\smallskip
\noindent
(b)
If all of the intermediate words $w_i$ in $W$ belong to a finite family 
$\cal W$, then the above constant $K$ can be taken to be equal to 
$(q-1) K_0$ for some constant $K_0 \geq 0$ which only depends on $\cal W$ and $\cal{T}$-length, but not on the multi-exponent $[n]$. 
Similarly, the constant $n_0$ can be chosen to depend only on $\cal W$.
\end{prop}

\begin{proof}
For the given family $W = W(w_i, y_i)$ we set 
$$\cal{F}(W)=\{(y, w, y') \mid w \in \cal W, \,  y, y' \in \cal T,  \,\, y^{-m} \neq w y'^{n} w^{-1} \,\, {\rm if} \,\, m, n \geq 0\}
$$ 
Then our claim follows directly from Proposition~\ref{prop1}.
\end{proof}

\section{Graph-of-groups with trivial edge groups: Growth bounds}

In this section we will suppress the base point $v$ in the fundamental group $\pi_1 \cal G$ of a graph-of-groups $\cal G$ if it is immaterial, and write simply $\pi_1 \cal G$.

\begin{lem}\label{lemma5}
Let $\cal{G}$ be a graph-of-groups with trivial edge groups. 
Let $(W_i)_{i=1}^{+\infty} \subset 
\pi_{1}(\cal{G})$
be a family of cyclically reduced words on $\cal{G}$, where $W_i = v_0(i) t_1 v_1(i) t_2 ... t_q v_q(i)$.
If for some $1 \leq k \leq q$, the length of $v_k(i)$ under some 
(hence any) 
finite generating system
$\cal{B}_k$ of the vertex group $G_{v_k}$, i.e. $|v_k(i)|_{\cal{B}_k}$, grows quadratically with respect to $i$, then $\|W_i\|$ grows at least quadratically with respect to $i$.
\end{lem}

\begin{proof}
As shown in Section~\ref{graphofgroups} that the union of 
generating systems 
of each vertex group $\bigcup\limits_{v \in V(\cal{G})} \cal{B}_{v}$ forms 
a subset of a generating system 
$\cal{B}$ of $\pi_1(\cal{G})$,
where each vertex group is a free factor.

We obtain:
$$\|W_i\|_{\cal B} \geq \sum\limits_{k=1}^{q-1}|v_k(i)|_{\cal{B}_{k}}+ |v_q(i) v_0(i)|_{\cal{B}_0}$$
It follows immediately from the conditions that at least one of $|v_k(i)|_{\cal{B}_k}$ grows quadratically and that $W_i$ is cyclically reduced that the cyclically reduced length of $W_i$ grows at least quadratically.

\end{proof}

Recall that a graph-of-groups $\cal G$ is called {\em minimal} if it doesn't contain a proper subgraph $\cal G'$ such that the inclusion induces an isomorphism on the fundamental groups. For a finite graph-of-groups with trivial edge groups this amounts to requiring that any vertex of valence 1 has a non-trivial vertex group.

\begin{lem}\label{lemma6}
Let $\cal{G}$ be a minimal graph-of-groups with trivial edge groups,
Then for any edge e
with terminal vertex $v = \tau(e)$ that has non-trivial vertex group $G_v$ one can find a cyclically reduced word $w \in \pi_1(\cal{G})$ with underlying path that runs subsequently through the edge $e$ 
and directly after through $\bar e$.
\end{lem}

\begin{proof}
Because $\cal{G}$ is a minimal graph-of-groups, each connected component of the graph $\Gamma'$ obtained from $\Gamma(\cal G)$ by removing the edge $e$ must contain either a circuit $\omega$ or else a vertex $v'$ with non-trivial vertex group $G_{v'}$. Let $\gamma$ be a path in $\Gamma'$ which connects $\iota(e)$ either to the initial (= terminal) vertex of some such $\omega$, or else to $v'$. Let $v \in G_v \smallsetminus \{1\}$ and $u \in G_{v'} \smallsetminus \{1\}$ (in the second case only). Then $\gamma_*^{-1} e v \bar e \gamma_* \omega_*$ (in the first case) or $\gamma_*^{-1} e v \bar e \gamma* u$ (in the second one) are the words we are looking for, where $\gamma_*$ and $\omega_*$ denote the sequence of stable letters $t_{e_i}$ defined by the edges $e_i$ of $\gamma$ and $\omega$ respectively.
\end{proof}

\begin{prop}
\label{free-product-growth}
Let $\cal G$ be a minimal graph-of-groups with trivial edge groups, and let $H: \cal G \to \cal G$ be a graph-of-groups automorphism which acts trivially on the underlying graph $\Gamma = \Gamma(\cal G)$. 

Let $v$ be a vertex of $\Gamma$, with vertex group automorphism $H_v: G_v \to G_v$.
For some edge $e$ with terminal vertex $\tau(e) = v$ denote by $\delta_e \in G_v$ the correction term of $e$.

Assume that for some $g \in G_v$ there exist a constant $C>0$ and a strictly increasing sequence of numbers $n_i \in \Z$
which satisfy:
$$|H_{v}^{(n_i)}(\delta_{e}^{-1}) H_v^{n_i}(g) H_{v}^{(n_i)}(\delta_{\bar e}^{-1})^{-1}| \geq C n_{i}^2$$
Then the induced outer automorphism $\hat{H}$ of $\pi_1\cal G$ has at least quadratic growth.
\end{prop}

\begin{proof}
%
If follows immediately from Lemma~\ref{lemma6} that one can find cyclically reduced word $w \in \pi_1\cal G$ with underlying path that runs through the edge $e$ and subsequently through $\bar e$, and $w$ contains the word $t_e g t_{e}^{-1}$ as subword.

As a consequence 
the iteration of $H_{*}: \Pi(\cal{G}) \rightarrow \Pi(\cal{G})$ on $w$ will give words $H_*^k(w)$ that contain 
$$t_e H_{v}^{(k)}(\delta_{e}^{-1}) H_v^{k}(g) H_{v}^{(k)}(\delta_{\bar e}^{-1})^{-1} t_{e}^{-1}$$
as subword.




Hence it follows from the assumed inequality and from Lemma~\ref{lemma5} that the subsequence $\|H_{*}^{n_i}(w)\|$ grows at least quadratically. Therefore the conjugacy class of $w$ grows at least quadratically under the iteration of $H_{*}$, which implies that the induced outer automorphism $\hat{H}$ grow at least quadratically. 
\end{proof}

\section{Dehn twists}
This section is dedicated to translate our cancellation propositions into graph-of-groups language.

Through the whole section, we always assume that the free group $F_n$ is of rank $n \geq 2$.

We first prove the following Proposition.

\begin{prop}
\label{Dehn-twist-lifts+}

Let $F_n$ be a free group
with rank $n \geq 2$, 
and let ${\cal{D}} \in \Aut(F_n)$ be a Dehn twist automorphism 
which is represented by an efficient Dehn twist. Then 
we have:


\smallskip
\noindent
(1)
There exists a finite set of ``twistors'' $\cal{T} = \{z_1, \ldots z_r\} \subset F_n \smallsetminus \{1\}$, 
such that for any element $w \in F_n$ there exists a (non-unique) ``$\cal T$-decomposition'' of $w$ as product
\begin{equation}
\label{T-product}
w = w_0 w_1 w_2 \ldots w_{q-1} w_q = w_0 y_{1}^0 w_1 y_{2}^0 w_2 \ldots w_{q-1} y_{q}^0 w_q
\end{equation}
with $w_i \in F_n$ and 
$y_i$ 
or $y_i^{-1}$
in $\cal{T}$ such that 
$$\cal{D}^n(w) = w_0 y^n_1 w_1 y^n_2 w_2 \ldots w_{q-1} y^n_q w_q \, ,$$
and $y_i^{-m} \neq w_i y_{i+1}^{m'} w_i^{-1}$ for any integers $m, m' \geq 0$. 

\smallskip
\noindent
(2)
The rank of the subgroup of $F_n$ which consists of all elements fixed by $\cal{D}$ satisfies:
$$ {\rm rk}({\rm Fix}(\cal{D})) \geq 2$$

\end{prop}

\begin{proof}
By definition ${\cal D}$ is represented by an efficient Dehn twist $D: \cal G \to \cal G$ on a graph-of-groups $\cal G$ with fundamental group $\pi_1 \cal G$ isomorphic to $\FN$. 
We pick a vertex $v_0$ of $\cal G$ as base point and specify the above isomorphism to be $\theta: \FN \overset{\cong}{\longrightarrow} \pi_1(\cal G, v_0)$. 
The automorphism $\cal D$ fixes the $\theta^{-1}$-image of the vertex group $G_{v_0}$ of $\cal G$ elementwise, and since efficient Dehn twists have all vertex groups of rank $\geq 2$ (see Proposition \ref{vertex-rank}), this proves claim (2) of the proposition.

In order to obtain claim (1), we chose a maximal tree $Y$ in the graph $\Gamma$ and identify in the usual fashion
each vertex group $G_v$ canonically with a subgroup of $\pi_1(\cal G, v_0)$ by connecting $v_0$ to $v$ through a simple path in $Y$. 
Similarly, for any edge $e$ the stable letter $t_e \in \Pi(\cal G)$ gives rise to an element in $\pi_1(\cal G, v_0)$ by connecting $v_0$ to the terminal vertices of $e$ through simple paths in $Y$ (which gives $1 \in \pi_1(\cal G, v_0)$ if and only if $e$ belongs to $Y$).

The collection $\cal T$ is then given by the twistors $z_e$ of the edges $e \in E^+(\cal G)$ 
(for some orientation $E^+(\cal G) \subset E(\cal G)$, see subsection \ref{graphofgroups}). For any $w \in \FN$ the collection of factors $w_i$ in the $\cal T$-product decomposition (\ref{T-product}) is obtained by writing $\theta(w)$ as a reduced word $v_0 t_1 v_1 t_2 v_3 \ldots v_{q-1} t_q v_q$ in $\pi_1(\cal G, v_0)$ (see Proposition \ref{reduced-words-existence}), and by applying $\theta^{-1}$ to $v_0$ or to any of the $t_i v_i$ for $i = 1, \ldots, q$. The equality $\cal{D}^n(w) = w_0 y^n_1 w_1 y^n_2 w_2 \ldots w_{q-1} y^n_q w_q $ is a immediate consequence of the definition of a Dehn twist automorphism,
see Remark \ref{rem-twistor}. The inequalities $y_i^{-m} \neq w_i y_{i+1}^{m'} w_i^{-1}$, for any integers $m, m' \geq 0$, follow directly from the condition that the efficient Dehn twist $D$ does not have twistors that are positively bonded, see Definition \ref{efficient-D-twists-defn} (5).

\end{proof}

\begin{rem}
The representation of $w \in \FN$ as {\em $\cal T$-product} as given in Proposition \ref{Dehn-twist-lifts+} is not unique. However, it follows from the proof that the sequence of twistors $y_i$ is well defined, and thus also the $\cal T$-length $|w|_{\cal T}$ of $w$. The intermediate words $w_i$ are well defined up to replacing them by $y_i^p w_i y_{i+1}^q$ for some $p, q \in \Z$.
\end{rem}

\begin{defn}
\label{cyclically-T}
A $\cal T$-product representative of $w \in \FN$ as in Proposition \ref{Dehn-twist-lifts+} is called {\em cyclically $\cal T$-reduced} if $y_q^{-m} \neq w_q w_0 y_{1}^{m'} w_0^{-1} w_q^{-1}$ for any integers $m, m' \geq 0$.
\end{defn}

\begin{defn-rem}
\label{D-reduced}
${}^{}$

\noindent
(1)
Let $\cal D$ and $\cal T$ be as in Proposition \ref{Dehn-twist-lifts+}. Assume that $D: \cal G \to \cal G$ is an efficient Dehn twist representative of $\cal D$, with respect to some identification isomorphism $\theta: \FN \to \pi_1(\cal G, v_0)$ for some vertex $v_0$ of $\cal G$.

For some element $w \in F_n$ let $W \in \Pi(\cal{G})$ be the corresponding element in the Bass group of $\cal G$, i.e. $W = \theta(w) \in \pi_1(\cal G, v_0) \subset \Pi(\cal G)$.

We say that $w$ is {\em $\cal D$-reduced} if $W$ is $D$-reduced.

\smallskip
\noindent
(2)
Recall that $\theta(w)$ is called $D$-reduced if the element $W$ is reduced as word in $\Pi(\cal G)$, and if its $\cal G$-length can not be shortened by $D$-conjugation, i.e. by passing over to a word $V^{-1} W D(V) \in \pi_1(\cal G, v_1) \subset \Pi(\cal G)$, for some vertex $v_1$ of $\cal G$. It follows from our considerations in the proof of Proposition \ref{Dehn-twist-lifts+} that in this case $w$ is $\cal T$-reduced and cyclically $\cal T$-reduced.

\smallskip
\noindent
(3)
If $W$ is not $D$-reduced, then we can follow the procedure 
indicated in Definition \ref{Hreduced-defn}  and Remark \ref{Hreduced}
to perform iteratively elementary $D$-reductions until we obtain a new word $W' \in \pi_1(\cal G, v_1) \subset \Pi(\cal G)$ which is $D$-reduced, for a possibly different vertex $v_1$. In this case we 
change our identification isomorphism $\theta$ corresponding to the performed $D$-reductions to obtain a new identification $\theta': \FN \to \pi_1(\cal G, v_1)$ with respect to which $w$ is $D$-reduced.

\smallskip
\noindent
(4)
Recall that a $D$-reduced word $W \in \Pi(\cal G)$ is $D$-zero if and only if $W$ has $\cal G$-length 0, or in other words, $W$ is contained in some vertex group $G_v$. In this case we say that a word $w \in \FN$ with $\theta(w) = W$ is {\em $\cal D$-zero}.

We thus note that any $w \in \FN$ which is $\cal D$-reduced (possibly with respect to a modified identification isomorphism $\theta'$ an in (3) above) but not $\cal D$-zero is $\cal T$-reduced, cyclically $\cal T$-reduced, and of $\cal T$-length $|w|_{\cal T} \geq 1$.



%
\end{defn-rem}

Let $\cal{D} \in Aut(F_n)$ be a Dehn twist automorphism, 
recall that for any element $w \in F_n$ and any integer $n \geq 1$ we denote by $\cal{D}^{(n)}(w)$ the {\em iterated product}, defined through:
$$\cal{D}^{(k)}(w) := w \, \cal{D}(w) \, \cal{D}^2(w) \, \ldots \, \cal{D}^{k-1}(w),$$
 
and the {\em partial iterated product} is given by 
$$\cal{D}^{(k_1, k_2)}(w) := \cal{D}^{k_1}(w) \, \cal{D}^{k_{1}+1}(w) \, \ldots \, \cal{D}^{k_2}(w).$$

\begin{prop}\label{quadratic1}
Let $\cal{D} \in Aut(F_n)$ be a Dehn twist automorphism, and denote by $\cal{T}$ the set of ``twistors" defined in Proposition~\ref{Dehn-twist-lifts+}. 
Let $w \in F_n$ be a $\cal{D}$-reduced word which is not $\cal{D}$-zero, i.e. $|w|_{\cal{T}} \neq 0$.
Then the combinatorial length of $\cal{D}^{(k)}(w)$ has quadratic growth with respect to $k$, i.e. one can find constant $C_1$ such that 
$$|\cal{D}^{(k)}(w)| \simeq C_1 k^2. $$
\end{prop}

\begin{proof}


It follows from Proposition~\ref{Dehn-twist-lifts+} that $w$ admits a $\cal{T}$-reduced decomposition 
$w = w_0 y_1^0 w_1 y_2^0 w_2 \ldots w_{q-1} y_q^0 w_q$
with $w_i \in F_n$ and 
$y_i \in \cal{T}$ which satisfies that 
$\cal{D}^{k}(w)=c_0y_1^{k}c_i \ldots c_{q-1}y_q^{k}c_q$.

It follows immediately from Proposition~\ref{T-word-length} that there exist constants $N_0 \geq 0$ and $K_0$ such that for $k \geq N_0$ we have:
$$|\cal{D}^{k}(w)| \geq \sum\limits_{i=1}^{q} k \| y_i \| + (q-1)K_0.$$

Since $w$ is $\cal{D}$-reduced, 
Remark~\ref{D-reduced} shows that the cancellation in between $\cal{D}^{i}(w)\cal{D}^{i+1}(w)$ is bounded by some constant $K_1$.

By taking $K_{0}'=\min\{K_0, K_1\}$, we obtain that, 
for $N_0 \leq k' \leq k$,
$$|\cal{D}^{(k',k)}(w)| \geq \sum\limits_{j=k'}^{k} \sum\limits_{i=1}^{q} j \| y_i \| + (kq-k'q-1)K_0^{\prime}.$$
Since $q \geq 1$, the cancellation bounds $(kq-k'q-1)K_0^{\prime}$ grows linearly with respect to $k$.

In particular, we may take $k'=N_0$ hence 
\begin{align*}
|\cal{D}^{(N_0,k)}(w)| 
& \geq \sum\limits_{j=N_0}^{k} \sum\limits_{i=1}^{q} j \| y_i \| + (kq-N_{0}q-1)K_0^{\prime}\\
& = \frac{k(k-N_0)}{2}\sum\limits_{i=1}^{q} \| y_i \| + (kq-N_{0}q-1)K_0^{\prime}
\end{align*}

and 
$$|\cal{D}^{(k)}(w)| \geq \frac{k(k-1)}{2}\sum\limits_{i=1}^{q} \| y_i \| + (kq-N_{0}q-1)K_0^{\prime}- |\cal{D}^{(N_0)}(w)|.$$

On the other hand, for all $k \in \Bbb{N}$ we always have 
$$|\cal{D}^{(k)}(w)|  \leq \sum\limits_{j=1}^{k}\sum\limits_{i=1}^{q} j |y_i|+ \sum\limits_{j=1}^{k}\sum_{i=0}^{q}|c_i| = \frac{k(k-1)}{2}\sum\limits_{i=1}^{q} |y_i| + k \sum\limits_{i=0}^{q}|c_i| $$

Together these two inequalities give $|\cal{D}^{(k)}(w)|  \simeq C_1 k^{2}$ for some $C_1 >0$.

\end{proof}

\begin{rem}\label{cyclical}
Let $N_0$ be as above.
Denote 
$$B=\cal{D}^{(N_0)}(w)=wD(w)D^2(w) \ldots D^{N_0-1}(w).$$ 
We can now find $N_1 > N_0 \in \Bbb{N}$ large enough so that the following two conditions hold:

\begin{enumerate}
\item[1.] 
Let $w_1$ be the prefix of $\cal{D}^{(N_0, N_1)}(w)$ which ends with $y_{q-1}$ such that $|B| \leq |w_1|$.

\item[2.]
The number $N_1$ is large enough so that $|y_{q}^{N_1-1}| \geq |w_1|+ |c_q| + |c_{q-1}|$.

\end{enumerate}

These two conditions do guarantee that the cyclic cancellation of $\cal{D}^{(N_1)}(w)$ cannot proceed into the subword 
$$w_1^{-1} \cal{D}^{(N_0, N_1)}(w) = c_{q-1} y_{q}^{t_0 } c_{q} \ldots c_{q-1}y_{q}^{N_1} c_q$$
further than $c_{q-1} y_{q}^{t_0 }$ on the left and $y_{q}^{N_1} c_q$ on the right, 
as otherwise a subword of length $|y_q|$ of $y_{q}^{t_0 }$ would cancel against a subword of $y_{q}^{N_1}$, 
which is impossible since $w^{-1} \neq v w v^{-1}$ in $F_n$, for any $w \neq 1$.

Hence we obtain from the above arguments
that also 
the cyclic length $\|\cal{D}^{(k)}(w)\|$ has quadratic growth with respect to $k$.

\end{rem}

\section{Main Result and Theorem}

The following is a slightly stronger version of what has been stated as Theorem \ref{q-growth} in the Introduction.

%




\begin{thm}
\label{main-result}
Let $\cal G$ be a minimal graph-of-groups with trivial edge groups, and let $H: \cal G \to \cal G$ be a graph-of-groups automorphism which acts trivially on the underlying graph $\Gamma = \Gamma(\cal G)$. 

Assume that for some vertex $v$ of $\Gamma$ the 
vertex group $G_v$ is a free group, that the 
induced vertex group automorphism $H_v: G_v \to G_v$ is a Dehn twist automorphism, and that for some edges $e$ of $\Gamma$ with endpoint $\tau(e) = v$ the $H$-correction term $\delta_e$ is not $H_v$-zero.

Then the induced outer automorphism $\hat{H}$ of $\pi_1\cal G$ has at least quadratic growth.
\end{thm}


\begin{proof}
By Propositon~\ref{free-product-growth} it suffices to show that for $w = \delta_e$ one can find an element $u \in G_v$ such that there is a subsequence of $|H_{v}^{(k)}(w^{-1}) H_v^k(u) H_{v}^{(-k)}(w)|$ which grows quadratically.

Since by assumption the correction term $w = \delta_e$ is not $H_v$-zero, we can apply Proposition \ref{quadratic1} to obtain that the length of the iterated product $H_{v}^{(-k)}(w)$ grows at least quadratically. 

One then uses Proposition \ref{Dehn-twist-lifts+} (2) to show that the subgroup ${\rm Fix}(H_v)$ of all fixed elements of $H_v$ has rank at least 2. This enables us to find via Lemma \ref{cancellation-blocker} an element $u \in {\rm Fix}(H_v)$ such that for infinitely many indices $n_i \geq 0$ the cancellation in the product $H_{v}^{(n_i)}(w^{-1}) H_v^{n_i}(u) H_{v}^{(-n_i)}(w) = H_{v}^{(n_i)}(w^{-1}) u H_{v}^{(-n_i)}(w)$ is uniformly bounded, so that from the previous paragraph we obtain that their lengths grow at least quadratically.


\end{proof}
 
Recall from Definition \ref{partial-Dehn-twists} (b)
that a partial Dehn twist $H: \cal G \to \cal G$ relative to a family of local Dehn twists is given, for any vertex $v$ of $\cal G$, through an identification isomorphism $G_v \cong \pi_1 \cal G_v)$ and a ``local'' Dehn twist representatives $D_v: \cal G_v \to \cal G_v$ which we can assume without loss of generality to be efficient. We say that for any edge $e$ of $\cal G$ the correction term $\delta_e$ is {\em locally zero} if it is $D_{\tau(e)}$-zero.

Then the last theorem gives directly:

\begin{cor}
\label{rel-D-twists}
Let $\hat \phi \in Out(\FN)$ be represented by a partial Dehn twist $H: \cal G \to \cal G$ relative to a family of local Dehn twists.
Assume that for some edge $e$ of $\cal G$ the correction term $\delta_e$ is not 
locally zero.

Then $\hat \phi$ has at least quadratic growth.
\qed

\end{cor}

As a final remark we want to point out that Corollary \ref{q-growth+} is indeed a direct consequence of the above corollary together with the main result of \cite{KY01}, which states that, in the situation of Corollary \ref{rel-D-twists}, if all of the correction terms for the edges of $\cal G$ are locally zero, then $H$ can be blown-up at the local Dehn twists to give a Dehn twist representative of the automorphism $\hat \phi$.

\end{document}